\documentclass{amsart}

\usepackage{xcolor}
\usepackage{enumitem}
\usepackage{amssymb}
\usepackage{amsfonts}
\usepackage{amsmath}
\usepackage{amsthm}
\usepackage[]{hyperref}

\newtheorem{theorem}{Theorem}[section]
\newtheorem{lemma}[theorem]{Lemma}
\newtheorem{corollary}[theorem]{Corollary}

\theoremstyle{definition}
\newtheorem{definition}[theorem]{Definition}

\theoremstyle{remark}
\newtheorem{remark}[theorem]{Remark}

\DeclareMathOperator{\rp}{\mathbb R P}
\DeclareMathOperator{\R}{\mathbb R}
\DeclareMathOperator{\pgl}{PGL}
\DeclareMathOperator{\vol}{Vol}
\DeclareMathOperator{\tr}{tr}
\DeclareMathOperator{\bary}{bar}
\DeclareMathOperator{\ent}{ent}
\DeclareMathOperator{\jac}{Jac}
\DeclareMathOperator{\vspan}{span}

\numberwithin{equation}{section}

\begin{document}

\title{Entropy rigidity and Hilbert volume}

\author{Ilesanmi Adeboye}
\address{Department of Math \& Comp. Sci., Wesleyan University, Middletown, CT 06459}
\email{iadeboye@wesleyan.edu}

\author{Harrison Bray}
\address{Department of Mathematics, University of Michigan, Ann Arbor, MI 48109}
\email{brays@umich.edu}

\author{David Constantine}
\address{Department of Math \& Comp. Sci., Wesleyan University, Middletown, CT 06459}
\email{dconstantine@wesleyan.edu}



\date{\today}

\dedicatory{}

\begin{abstract} 
For a closed, strictly convex projective manifold of dimension $n\geq 3$ that admits a hyperbolic structure, we show that the ratio of Hilbert volume to hyperbolic volume is bounded below by a constant that depends only on dimension. We also show that for such spaces, if topological entropy of the geodesic flow goes to zero, the volume must go to infinity. These results follow from adapting Besson--Courtois--Gallot's entropy rigidity result to Hilbert geometries.
\end{abstract}

\maketitle

\thispagestyle{empty}

\section{Introduction}\label{sec:intro} 

A\emph{(strictly) convex real projective orbifold} is a quotient $\Omega/\Gamma$, where $\Omega$ is an open, properly (strictly) convex subset of $\rp^n$  and $\Gamma<\pgl(n+1,\R)$ is a discrete subgroup of projective transformations that preserves $\Omega$. A subset $\Omega\subset\rp^n$ is \emph{proper} if it is bounded in some affine patch; \emph{convex} if its intersection with any projective line is connected; and \emph{strictly convex} if, moreover, its topological boundary in an affine patch does not contain an open line segment. An orbifold  is a \emph{manifold} if $\Gamma$ contains no elements of finite order.

Any properly convex set $\Omega$ admits a complete Finsler metric called the \emph{Hilbert metric}. This is the Klein model of the hyperbolic metric when $\Omega$ is the interior of a round ball. Hence, the first examples of projective orbifolds are hyperbolic orbifolds. By Mostow rigidity, a hyperbolic structure on a closed manifold of dimension greater than or equal to 3 is unique, up to isometry. However, the dimension of the deformation space of strictly convex projective structures on some closed manifolds can be large. For instance, in dimension two, there is a $16g-16$ dimensional deformation space of strictly convex real projective structures on a closed surface of genus $g$ \cite{goldman90}. This space need not contain a hyperbolic point -- there exist closed strictly convex manifolds that do not admit a hyperbolic structure.

It is of interest to characterize a hyperbolic structure, when it exists, among all strictly convex projective structures a manifold admits. This article addresses that question in terms of volume and entropy and derives a pair of results on the geometry and dynamics of these spaces. The  Finsler structure on $\Omega$ provides a natural volume form on $Y$ referred to as \emph{Hilbert volume}. Let $\vol(Y,g_0)$ and $\vol(Y,F_{\Omega})$ denote the hyperbolic volume and the Hilbert volume, respectively, of a closed manifold $Y$. Our main result on Hilbert volume is given below.

\begin{theorem}\label{Main}
Let $Y$ be a closed strictly convex projective manifold of dimension $n\geq 3$ which admits a hyperbolic structure. Then there exists a constant $\mathcal D>0$, depending only on dimension, such that
\[ \frac{\vol(Y,F_{\Omega})}{\vol(Y,g_0)}\geq\mathcal D. \]
\end{theorem}

A consequence of the Margulis lemma \cite{Ballman1} is that there exists a positive lower bound for the volume of a hyperbolic $n$-manifold, for each $n$. This gives the following corollary.

\begin{corollary} 
Let $Y$ be a closed strictly convex projective manifold of dimension $n\geq 3$ which admits a hyperbolic structure. Then, there exists a constant $\mathfrak e>0$, depending only on dimension, such that 
\[ \vol(Y,F_{\Omega})\geq\mathfrak{e}.\]
\end{corollary}

Results bounding Hilbert area in dimension 2 can be found in \cite{CVV} and \cite{adeboye_cooper}.

It is important to underscore that there are strictly convex real projective manifolds which do not admit a hyperbolic metric. Coxeter group examples exist in dimension four \cite{Ben5} and Gromov-Thurston examples exist for each dimension greater than three \cite{Kapovich2007}. However, our theorem holds in several contexts.

%In dimension two, there is a $16g-16$ dimensional deformation space of strictly convex real projective structures on a closed surface of genus $g$ \cite{goldman90}. 

Every closed strictly convex projective 3-manifold admits a hyperbolic structure; by Benoist's dichtomoy \cite[Theorem 1.1]{Ben1}, strict convexity of $\Omega$ is equivalent to Gromov hyperbolicity of $\Gamma$ which implies the quotient admits a hyperbolic structure by geometrization. There are also examples of \textit{flexible} closed hyperbolic 3-manifolds that have nontrivial projective deformations \cite{CooperLongThistlethwaite2008}. In the same work, Cooper--Long--Thistlethwaite conjecture that all hyperbolic 3-manifolds are virtually flexible. Furthermore, any hyperbolic $n$-manifold with a totally geodesic submanifold admits the projective bending deformation of Thurston \cite{johnsonmillson}.

Our second main result concerns the dynamics of the geodesic flow for $Y=\Omega/\Gamma$.

\begin{definition} 
Let $g$ be a (Finsler or Riemannian) metric on a compact manifold $Y$. Let $y\in\tilde{Y}$ be any point in the universal cover of $Y$ and $B_g(R,y)$ the radius $R$ ball around $y$ with respect to $g$. The \emph{volume growth entropy} of $g$ is, 
\[ h(g)=\lim_{R\to\infty}\frac{1}{R}\log(\vol B_g(R,y)). \] 
\end{definition}

For a nonpositively curved Riemannian metric, the volume growth entropy is equal to the topological entropy for the geodesic flow \cite{manning}. This result can be generalized to non-Riemannian settings under some mild conditions mimicking nonpositive curvature (see \cite{leuzinger}). Verification of these conditions in the present setting can be found in \cite[\S 8]{crampon09}.

We prove the following relationship between this dynamical quantity and the Hilbert volume:

\begin{theorem}\label{thm:entropytozero}
Let $Y_t= \Omega_t/\Gamma_t$ be a family of strictly convex real projective structures on a manifold of dimension at least 2 which supports a hyperbolic metric. Then 
\[h(F_{\Omega_t})\to 0 \ \Rightarrow \ \vol(Y_t,F_{\Omega_t})\to\infty.\] 
\end{theorem}

\begin{remark} 
Examples of manifolds in dimensions 2, 3, and 4, for which the entropy of strictly convex projective structures goes to zero are given in \cite{nie}. In these cases, volume is known to grow without bound as entropy decreases. Moreover, in dimension 2, Zhang proves the entropy of strictly convex real projective structures on any closed surface can be made arbitrarily small \cite{zhang15}. Though it is plausible the work of Zhang shows the volume will simultaneously diverge to infinity, it is not immediate. Theorem \ref{thm:entropytozero} states that this phenomenon will hold generally in any dimension with a short proof. 
\end{remark}

Theorem \ref{Main} and the $n\geq 3$ case of Theorem \ref{thm:entropytozero} are consequences of an entropy rigidity theorem. This theorem follows a line of results beginning with the celebrated work of Besson, Courtois and Gallot in \cite{BCG1, BCG2}. They prove the following theorem using the `barycenter method' -- the technique we will also use.

\begin{definition} 
The \emph{normalized entropy functional} of $(Y^n,g)$ is the quantity \[\ent(Y,g)=h(g)^n\vol(Y,g).\]
\end{definition}

\begin{theorem}[see Th\'eor\`eme Principal \cite{BCG1}]\label{thm:BCG}
If $(Y,g)$ is a compact, oriented, Riemannian manifold of dimension $n\geq 3$ homotopy equivalent to a negatively curved locally symmetric space $(X,g_0)$, then 
\[ \ent(Y,g)\geq \ent(X, g_0) \]
with equality if and only if $(Y,g)$ and $(X,g_0)$ are isometric, up to a homothety.
\end{theorem}

\begin{remark} 
Theorem \ref{thm:BCG} has a number of important consequences, including a proof of Mostow's rigidity theorem (see \cite{BCG1, BCG2}). The barycenter method has been employed many times, including the work of Connell and Farb on higher-rank symmetric spaces \cite{CFdegree, CFproducts}. See \cite{CFrecent} for a survey. 
\end{remark}

Our paper closely follows the work of Boland and Newberger in \cite{boland-newberger}, which adapts the Besson--Courtois--Gallot result to compact Finsler manifolds of negative flag curvature. For a $C^2$-Finsler metric $F$ on a manifold, Boland and Newberger define the \emph{eccentricity factor}, denoted by $N(F)$. See Section \ref{sec:natural} for the definiton, but note here that $N(F)$ is equal to 1 when $F$ is Riemannian and is strictly greater than 1 otherwise. (The terminology `eccentricity factor' is coined in \cite{CFrecent}.) Their Finsler extension of Theorem \ref{thm:BCG} is as follows.

\begin{theorem}[see Main Theorem \cite{boland-newberger}]\label{thm:BN} 
Let $(M,F)$ be a compact, reversible, $C^2$-Finsler manifold of negative flag curvature and dimension $\geq 3$ with the same homotopy type as the compact, negatively curved, locally symmetric space $(X,g_0)$. Then
\begin{enumerate}[label=(\roman*)]
	\item $\ent(X,g_0) \leq N(F)\ent(M,F).$
	\item Equality holds above if and only if $(M,F)$ is Riemannian and homothetic to $(X,g_0)$.
\end{enumerate}
\end{theorem}

We extend this result to the Hilbert geometry setting:

\begin{theorem}\label{thm:rigidity}
Let $(Y,F_\Omega)$ be a compact strictly convex real projective manifold of dimension $\geq 3$. Let $(X, g_0)$ be a hyperbolic structure on the same underlying manifold. Then there is a number $N(F_\Omega)\geq 1$, such that 
\[ N(F_\Omega)\ent(Y,F_\Omega)\geq \ent(X, g_0),\]
with equality if and only if $(Y, F_\Omega)$ is isometric to $(X, g_0)$.
\end{theorem}

Our modifications to the work of Boland and Newberger revolve around the following point: If $F_{\Omega}$ is a Finsler metric defined on a strictly convex domain $\Omega\subset\mathbb R\mathrm P^n$, one can verify (see Section \ref{sec:hilbert}) that $F_{\Omega}$ is $C^2$ if and only if $\partial\Omega$, the boundary of $\Omega$, is $C^2$.  If $\partial \Omega$ is $C^2$, then $\partial\Omega$ is in fact an ellipsoid (originally due to \cite[Theorem C]{benzecri}, see also \cite[Section 3.2]{cramponnotes} for an expository note in English). Hence, any corresponding $Y=\Omega/\Gamma$ with the induced metric is hyperbolic. If $F_{\Omega}$ is not $C^2$, then $\partial\Omega$ is only $C^{1+\alpha}$ for some $0<\alpha<1$ \cite[Theorem 1.3]{Ben1}. The failure of $C^2$ regularity in general leads us to substitute the Blaschke metric, a particular Riemannian metric associated to $\Omega$, for the family of reference metrics used in \cite{boland-newberger} at a key point in the proof. The $C^2$ rigidity for Hilbert geometries also allows us to reach the rigidity conclusion of Theorem \ref{thm:rigidity} with a shorter argument.

\begin{remark} It is conjectured in \cite{BCG2} that Theorem~\ref{thm:BCG} remains true in the class of Finsler metrics. This is equivalent to a restatement of Theorem~\ref{thm:BN} without the presence of the $N(F)$ factor. Following suite, we make the conjecture that Theorem~\ref{thm:rigidity} is valid without the $N(F_\Omega)$ term.\end{remark}

%Theorems \ref{Main} and \ref{thm:entropytozero} for $n\geq 3$ are then consequences of Theorem \ref{thm:rigidity}together with the relationship, due to Benoist and Hulin \cite{benoisthulin13}, between the Hilbert metric and the Blaschke metric. 

%
%%
%%%
%%%%%%%%%%%%%%%%%%%%%%%%%%%%%%

\subsection*{Outline of the paper}

Section \ref{sec:pre} provides the necessary background information on the Hilbert metric and Hilbert volume. 

Section \ref{sec:1.9} contains the proof of Theorem~\ref{thm:rigidity}. The `natural map' between the Hilbert geometry and its hyperbolic counterpart is constructed in Section~\ref{subsec:natural}. The Jacobian of the natural map is bounded, and the inequality statement of Theorem~\ref{thm:rigidity} is deduced in Section~\ref{sec:Jacobian}. Finally, the rigidity statement of Theorem~\ref{thm:rigidity} is proved in Section~\ref{sec:rigidity}.

Section~\ref{sec:1.1} recalls basic properties of the Blaschke metric for Hilbert geometries, in particular a relationship between the Hilbert and Blaschke metrics due to Benoist and Hulin \cite{benoisthulin13}. We then prove Theorem \ref{Main} and the $n\geq 3$ case of Theorem \ref{thm:entropytozero}. We conclude by proving Theorem \ref{thm:entropytozero} for 2-dimensional Hilbert geometries. Since our previous results require dimension greater than two, this argument uses the well-known result of Katok on entropy rigidity for surfaces \cite{katok_entropy}, as well as the particularly nice behavior of the Blaschke metric in dimension 2 due again to Benoist--Hulin \cite{benoisthulin14}.

%
%%
%%%
%%%%%%%%%%%%%%%%%%%%%%%%%%%%%%%

\subsection*{Acknowledgments}

The authors are grateful to Dick Canary, Daryl Cooper and Ralf Spatzier for many useful discussions. The second author was supported in part by NSF RTG grant 1045119.

%
%%
%%%
%%%%
%%%%%
%%%%%%%%%%%%%%%%%%%%%%%%%%%%%%%%%%

\section{Hilbert Geometry}\label{sec:pre}

This section provides the basic definitions for Hilbert geometry. For more details the reader may consult \cite{BK} and \cite{CLTi}.

%
%%
%%%
%%%%%%%%%%%%%%%%%%%%%%%%%%%%%%%%%%

\subsection{Definition of a Hilbert geometry}\label{sec:hilbert}

Let $\Omega$ be a properly convex domain of $\mathbb R\mathrm P^n$, as defined in the Introduction. Let $\partial\Omega$ denote the boundary of $\Omega$. The \emph{Hilbert metric} $d_{\Omega}$ on $\Omega$ is given by
\[ d_{\Omega}(x,y)=\frac12\bigg\lvert\log[p,x,y,q]\bigg\rvert=\frac12\bigg\lvert\log \left(\frac{\|y-p\|\cdot\|x-q\|}{\|y-q\|\cdot\|x-p\|}\right)\bigg\rvert \]
where $p,q$ are the intersection points, in the chosen affine patch, of $\partial\Omega$ with the projective line containing $x$ and $y$, and $\|\cdot\|$ denotes Euclidean distance.

Elements of $\pgl(n+1,\R)$ that preserve $\Omega$ are isometries of the Hilbert metric, since projective transformations preserve cross-ratio. If $\Omega$ is strictly convex, these elements constitute the full group of isometries, denoted by $\pgl(\Omega)$ \cite{H}. A \emph{Hilbert geometry} is a triple
\[(\Omega,d_{\Omega},\pgl(\Omega)).\] 

The Hilbert geometry where $\Omega$ is the interior of an ellipsoid is the Klein model of hyperbolic $n$-space, $\mathbb H^n$. The factor of $1/2$ in the definition of $d_{\Omega}$ ensures constant curvature $-1$.

The Hilbert metric on a properly convex domain $\Omega$ induces the Finsler structure
\[F_{\Omega}(y,v)=\frac{\|v\|}{2}\left(\frac{1}{\|y-v^-\|}+\frac{1}{\|y-v^+\|}\right) \] 
where $v^-,v^+$ are the points of intersection of the projective line through $y$ in the direction of $v$ with $\partial\Omega$. Note that it is immediate from the definition that $C^r$-regularity of $F_\Omega$ is equivalent to $C^r$-regularity of $\partial\Omega$.

%
%%
%%%
%%%%%%%%%%%%%%%%%%%%%%%%%%%%%%%%%%%%%

\subsection{Hilbert Volume} 

A group $\Gamma<\pgl(\Omega)$ is discrete in $\pgl(n+1,\R)$ if and only if $\Gamma$ acts properly discontinuously on $\Omega$ \cite{CLTi}. Therefore, for such a $\Gamma$ the projective and Finsler structures of $\Omega$ descend to the quotient orbifold $Y=\Omega/\Gamma$.

The Finsler structure on $Y$ provides for a natural definition of volume. Fix any Riemannian metric $g$ on $Y$ and let $B_g(1,y)$ and $B_{F_{\Omega}}(1,y)$ denote balls of radius 1 in $T_yY$ with respect to $g$ and $F_{\Omega}$, respectively. Then for any $y\in Y$, the Finsler volume element is 
\[dF_{\Omega}(y)=\frac{\vol_g(B_g(1,y))}{\vol_g(B_{F_{\Omega}}(1,y))}dg. \]
It is easy to check that the definition is independent of the choice of $g$. The volume of $Y$ with respect to Finsler volume will be referred to as the \emph{Hilbert volume} of $Y$.

%
%%
%%%
%%%%
%%%%%
%%%%%%%%%%%%%%%%%%%%%%%%%%%%%%%%%%%%%%%%%

\section{Proof of Theorem~\ref{thm:rigidity}}\label{sec:1.9}

In this section we prove Theorem~\ref{thm:rigidity}. As mentioned in the Introduction, the argument closely follows Boland--Newberger's adaptation of the Besson--Courtois--Gallot result to the Finsler setting. For completeness, we present the argument here, highlighting the modifications we make and referring the reader to the original papers for the details which are unchanged.

%
%%
%%%
%%%%%%%%%%%%%%%%%%%%%%%%%%%%%%%%%%%%%%%%

\subsection{The natural map}\label{subsec:natural} 

Let $Y=\Omega/\Gamma$ be a compact strictly convex projective $n$-manifold, $n\geq 3$. Let $X$ be a hyperbolic manifold homeomorphic to $Y$. Note that $X=\mathcal E/\Gamma_0$, where $\mathcal E$ represents an ellipsoid, $\Gamma_0<\pgl(\mathcal E)$, and $\Gamma_0\cong\pi_1(X)\cong\pi_1(Y)\cong\Gamma$. The universal covers have natural identifications: $\tilde{Y}=\Omega$ and $\tilde{X}=\mathcal E=\mathbb H^n$ is hyperbolic $n$-space.

%
%%
%%%
%%%%%%%%%%%%%%%%%%%%%%%%%%%%%%%%%%%%%%%%

\subsubsection{Busemann functions} 

Let $\Omega$ be a properly convex domain equipped with Hilbert metric $d_{\Omega}$. For $p,z\in\Omega$ the \emph{Busemann function} $B_{p,z}^{\Omega}:\Omega\rightarrow\mathbb R$ is defined by 
\[B^{\Omega}_{p,z}(q)=d_{\Omega}(p,q)-d_{\Omega}(q,z).\]
If $\partial\Omega$ is $C^1$, the Busemann function can be extended so that the argument $z$ takes values in  $\overline{\Omega}$ (cf. \cite[Lemme 3.4]{Ben1}). For $\xi\in\partial\Omega$, take \[B^{\Omega}_{p,\xi}(q):=\lim_{z\to\xi}B^{\Omega}_{p,z}(q)\] where $z\to\xi$ along any path. Geometrically, $B^{\Omega}_{p,\xi}(q)$ is the signed distance between horospheres based at $\xi$ passing through $p$ and $q$. If $p$ is fixed as a basepoint, then $B^{\Omega}_{p,\xi}(q)$ can be viewed as a family of functions mapping $\Omega$ to $\mathbb R$ that is parametrized by elements in $\partial \Omega$. With this notation, Busemann functions on $\tilde X$ will be denoted by $B^\mathcal{E}_{p,\xi}$.

%
%%
%%%
%%%%%%%%%%%%%%%%%%%%%%%%%%%%%%%%%%%%%%%%

\subsubsection{Patterson-Sullivan measures} 

The \emph{visual boundary} of a convex domain $\Omega$ is the space of all geodesic rays based at a point modulo bounded equivalence. If $\Omega$ is strictly convex with $C^1$-boundary, the visual boundary of  $\Omega$ coincides with $\partial \Omega$. 

Suppose, furthermore, that $\Omega$ admits a cocompact action by a discrete group of projective transformations. In this case we can define the \emph{Patterson-Sullivan} density, a family of measures $\{\mu_p\}_{p\in\Omega}$ on the boundary of $\Omega$. The defining properties of the Patterson-Sullivan density are the following:
\begin{itemize}
\item (quasi-$\Gamma$-invariance) $\mu_{\gamma p}=\gamma_*\mu_p$ for all $\gamma\in\Gamma$, $p\in\Omega$
\item (transformation rule) $\frac{d\mu_q}{d\mu_p}(\beta)=e^{-hB^{\Omega}_{p,\beta}(q)}$
\end{itemize}
where $h$ is the topological entropy of the Hilbert geodesic flow or, equivalently in our setting, the volume entropy of $(\Omega,F_{\Omega})$. 

%{\color{red} since $\Gamma$ acts cocompactly by isometries it is actually fairly immediate that the critical exponent agrees with the volume entropy via a proof that would be familiar to the experts, so I will omit a remark on the critical exponent.}

The construction of the Patterson-Sullivan measures originates with the work of Patterson for Fuchsian groups \cite{patterson} and Sullivan for convex cocompact actions on spaces of constant negative curvature \cite{sullivan79}. The concept has been extended to many settings, the most relevant being compact negatively curved manifolds and $\text{CAT}(-1)$ metric spaces \cite{Kaimanovich1990, roblin}. For the familiar reader, we remark that although non-Riemannian Hilbert geometries are not negatively curved in the classical sense and are not even $\text{CAT}(0)$, extending the Patterson-Sullivan theory in the strictly convex case is straightforward (c.f. \cite[Section 4.2]{cramponthese}).

%
%%
%%%
%%%%%%%%%%%%%%%%%%%%%%%%%%%%%%%%%%%%%%%%

\subsubsection{The barycenter of a measure} 

Fix some basepoint $o\in\tilde{X}$. For any probability measure $\lambda$ on $\partial\tilde{X}$ and $x\in\tilde{X}=\mathcal E$, let
\[\mathcal B(x,\lambda):=\int_{\alpha\in\partial\tilde{X}} B_{o,\alpha}^{\mathcal E}(x)d\lambda(\alpha).\]
The Busemann functions on $\tilde{X}$ are strictly convex along geodesic segments, hence $\mathcal B(x,\lambda)$ has a unique minimum \cite[Appendix A]{BCG1}. Denote this minimum by $\bary(\lambda)$; this is the \emph{barycenter} of $\lambda$.

It is a straightforward exercise to check that the barycenter of $\lambda$ is $\Gamma$-equivariant, that is, that
\[\bary(\gamma_*\lambda)=\gamma\cdot\bary(\lambda) \text{ for all } \gamma\in\Gamma.\]

%
%%
%%%
%%%%%%%%%%%%%%%%%%%%%%%%%%%%%%%%%%%%%%%%

\subsubsection{The natural map}\label{sec:natural}

Let $f:\partial\tilde{Y}\rightarrow\partial\tilde{X}$ be the $\Gamma$-equivariant homeomorphism induced by the identification of fundamental groups. A natural $\Gamma$-equivariant map from $\tilde{Y}$ to $\tilde{X}$ is constructed by associating to each $y\in\tilde{Y}$ the barycenter in $\tilde{X}$ of the push-forward of the Patterson-Sullivan measure $\mu_y$ under the map $f$. That is, $\tilde{\Phi}:\tilde{Y}\rightarrow\tilde{X}$ is given by
\[\tilde{\Phi}(y)=\bary(f_*\mu_y).\]

The $\Gamma$-equivariance of $\tilde{\Phi}$ follows from the $\Gamma$-equivariance of $f$ and $\bary$, and so it descends to a map $\Phi:Y\rightarrow X$. This `natural map' is at the heart of the Besson-Courtois-Gallot approach to entropy rigidity. Theorem~\ref{thm:rigidity} will be proved by bounding the Jacobian of $\tilde{\Phi}$.

\begin{remark} Boland and Newberger assume their Finsler manifold has negative flag curvature. This ensures that $\tilde{Y}$ is diffeomorphic to $\mathbb R^n$. In our setting, $\tilde{Y}$ is equal to $\Omega$, a bounded domain in projective space, by assumption.
\end{remark}

%
%%
%%%
%%%%%%%%%%%%%%%%%%%%%%%%%%

\subsection{The Jacobian of the natural map}\label{sec:Jacobian}

Let $v(x,\alpha)$ be the hyperbolic unit tangent vector based at $x\in\tilde{X}$ with forward endpoint $\alpha\in\partial\tilde{X}$. Note that $\bary(\lambda)=\hat{x}$ if and only if $d_{\hat{x}}\mathcal{B}(x,\lambda)=0$ where $d$ denotes the gradient. Since
\[d\mathcal{B}(x,\lambda)=\int_{\alpha\in\partial\tilde{X}}dB^{\mathcal E}_{o,\alpha}(x)d\lambda(\alpha)\]
 and, by the geometric description of the Busemann functions in hyperbolic space, $dB^{\mathcal E}_{o,\alpha}(x)=-v(x,\alpha)$, we see that $\bary(\lambda)$ is characterized implicitly by the condition
 \[0=\int_{\alpha\in\partial\tilde{X}}v(\bary(\lambda),\alpha)d\lambda(\alpha).\]
 Therefore, $\tilde{\Phi}(y)$ satisfies
 \[0=\int_{\alpha\in\partial\tilde{X}}v(\tilde{\Phi}(y),\alpha)d(f_*\mu_y)(\alpha).\]

 Changing variables by setting $\beta=f^{-1}(\alpha)$, and by the transformation rule, we have for a fixed $p\in\tilde{Y}$ and for all $y\in\tilde{Y}$, \[0=\int_{\beta\in\partial\tilde{Y}}v(\tilde{\Phi}(y),f(\beta))e^{-h(F_{\Omega})B^{\Omega}_{p,\beta}(y)}d\mu_p(\beta).\]

This expression allows us to verify that $\tilde \Phi$ is differentiable.  Let
 \[F(x,y)=\int_{\beta\in\partial\tilde{Y}}v(x,f(\beta))e^{-h(F_{\Omega})B^{\Omega}_{p,\beta}(y)}d\mu_p(\beta).\]
Clearly $v$, and therefore $F$, is smooth in its first variable. Since $\partial\Omega$ is $C^{1+\alpha}$, $B^\Omega_{p,\beta}(y)$, and therefore $F$, is differentiable in $y$ (see \cite[$\S$3.24]{Ben1}). $F$ implicitly defines $\tilde{\Phi}$ by $F(\tilde{\Phi}(y),y)=0$ and the Implicit Function Theorem implies that $\tilde \Phi$ is differentiable.

We compute the Jacobian of $\tilde \Phi$ as follows. Let $\langle - , - \rangle$ denote the inner product with respect to $g_0$.  For our fixed $p\in\tilde{Y}$, all $y\in\tilde{Y}$, and all $w\in T_{\tilde{\Phi}(y)}\tilde{X}$,
\[0=\int_{\beta\in\partial\tilde{Y}}\langle v(\tilde{\Phi}(y),f(\beta)),w\rangle e^{-h(F_{\Omega})B^{\Omega}_{p,\beta}(y)}d\mu_p(\beta).\]
Taking the differential of this expression with respect to $y$ we have
\begin{align}\label{eqn:ints}
\begin{split}
	\int_{\beta\in\partial\tilde{Y}}\langle Dv_{(\tilde{\Phi}(y),f(\beta))} & D_y\tilde{\Phi}(u),w\rangle d\mu_y(\beta) \\
	& h(F_{\Omega})\int_{\beta\in\partial\tilde{Y}}\langle v(\tilde{\Phi}(y),f(\beta)),w\rangle D_yB^{\Omega}_{p,\beta}(u)d\mu_y(\beta)
	\end{split}
\end{align}
for all $w\in T_{\tilde{\Phi}(y)}\tilde{X}$ and all $u\in T_y\tilde{Y}$.

Let $K$ and $H$ be the endomorphisms on $T_{\tilde{\Phi}(y)}\tilde{X}$ defined by
\[\langle K(w^{\prime}),w\rangle=\int_{\alpha\in\partial\tilde{X}}\langle Dv_{(\tilde{\Phi}(y),\alpha)}w^{\prime},w\rangle d(f_*\mu_y)(\alpha)\] 
and 
\[\langle H(w),w\rangle=\int_{\alpha\in\partial\tilde{X}}\langle v(\tilde{\Phi}(y),\alpha),w\rangle^2 d(f_*\mu_y)(\alpha).\]

The reader can verify (or see the nice exposition in \cite{Feres}) that $H$ and $K$ are symmetric, $\tr(H)=1$, and (since $\tilde{X}$ is hyperbolic space) that $K=I-H$. Then from equation~\ref{eqn:ints} and an application of Cauchy-Schwarz, we have
\begin{align}\label{eqn:CS}
	\begin{split}
		\lvert\langle K\circ D_y\tilde{\Phi}(u),w\rangle\rvert=&h(F_{\Omega})\bigg\lvert\int_{\beta\in\partial\tilde{Y}}\langle v(\tilde{\Phi}(y),f(\beta)),w\rangle D_yB^{\Omega}_{p,\beta}(u)d\mu_y(\beta)\bigg\rvert\\
		\leq& h(F_{\Omega})\langle H(w),w\rangle^{\frac{1}{2}}\left(\int_{\beta\in\partial\tilde{Y}} D_yB^{\Omega}_{p,\beta}(u)^2d\mu_y(\beta)\right)^{\frac{1}{2}}.
	\end{split}
\end{align}

Since the goal is to bound $\lvert\jac(\Phi)\rvert$ from above, we assume without loss of generality that $D_y\tilde{\Phi}$ has full rank. Fix a basis $\{e_i\}$ which is orthonormal with respect to the hyperbolic metric and in which the matrix for $H$ is diagonal. Fix a Riemannian metric $g_r$ on $Y$; we will specify a choice  for this metric in Section~\ref{sec:1.1}.

Let $v_i^{\prime}=(K\circ D_y\tilde{\Phi})^{-1}(e_i)$ and apply Gram-Schmidt to get a basis $\{v_i\}$ for $T_y\tilde{Y}$ which is orthogonal with respect to $g_r$. Then
\[ K\circ D_y\tilde{\Phi}:T_y\tilde{Y}\rightarrow T_{\tilde{\Phi}(y)}\tilde{X} \] 
is upper-triangular with respect to this basis.

We have
\[\jac(\tilde{\Phi})(y)=\frac{\det D_y\tilde{\Phi}}{\vol_{F_{\Omega}}(\vspan\{v_i\})} \] 
where the determinant is computed for the matrix with respect to the bases $\{e_i\}$ and $\{v_i\}$. Since $\{v_i\}$ is orthonormal for $g_r$, $\vspan\{v_i\}$ has $g_r$-volume 1. Hence, by the definition of $dF_{\Omega}$,
\[ \vol_{F_{\Omega}}(\vspan\{v_i\})=\frac{\vol_{g_r}(B_{g_r}(1,y))}{\vol_{g_r}(B_{F_{\Omega}}(1,y))}:=\rho(y,g_r).\] 
Thus we have that
\[ \lvert\jac(\tilde{\Phi})(y)\rvert=\frac{\lvert\det D_y\tilde{\Phi}\rvert}{\rho(y,g_r).} \]

Since 
\[ \lvert\det (K\circ D_y\tilde \Phi)\rvert = |\jac(\tilde \Phi)(y)|\cdot \rho(y,g_r) \cdot |\det K|, \]
we can use equation \eqref{eqn:CS}, the fact that $K\circ D_y\tilde \Phi$ is upper-triangular, and the fact that $H$ is diagonal to compute as follows: 
\begin{align}\label{eqn:JK bound}
	\begin{split}
		|\jac(\tilde \Phi&)(y) |\cdot \rho(y,g_r) \cdot|\det K| = \prod_{i=1}^n |\langle K\circ D_y\tilde \Phi(v_i),e_i\rangle|\\
			& \leq h(F_\Omega)^n \prod_{i=1}^n \langle H(e_i),e_i\rangle^{\frac{1}{2}} \prod_{i=1}^n \Big( \int_{\beta \in \partial \tilde Y} D_yB^\Omega_{p,\beta}(v_i)^2 d\mu_y(\beta)\Big)^{\frac{1}{2}}  \\
			& = h(F_\Omega)^n \det(H)^{\frac{1}{2}} \Big[ \big( \prod_{i=1}^n \int_{\beta \in \partial \tilde Y} D_yB^\Omega_{p,\beta}(v_i)^2 d\mu_y(\beta) \big)^{\frac{1}{n}} \Big]^{\frac{n}{2}}  \\
			& \leq h(F_\Omega)^n \det(H)^{\frac{1}{2}} \Big( \frac{1}{n} \sum_{i=1}^n \int_{\beta \in \partial \tilde Y} D_yB^\Omega_{p,\beta}(v_i)^2 d\mu_y(\beta) \Big)^{\frac{n}{2}}.
	\end{split}
\end{align}
At the last step we use the fact that the arithmetic mean bounds the geometric mean above. Since the $\{v_i\}$ are orthonormal for $g_r$, 
\begin{align}
	\begin{split}
		 \sum_{i=1}^n D_yB^\Omega_{p,\beta}(v_i)^2 = \Vert d_yB^\Omega_{p,\beta} \Vert_{g_r}^2 & = \max_{v\in S_{g_r}(1,y)} D_yB^\Omega_{p,\beta}(v)^2  \\
			 & = \max_{v\in S_{g_r}(1,y)} F_\Omega(v)^2 D_yB^\Omega_{p,\beta}(\hat v)^2 \label{eqn:maxFOmega2}
	 \end{split}
\end{align}
for $\hat v = \frac{v}{F_\Omega(v)}.$ Combining this with equation \eqref{eqn:JK bound}, and noting that for an $F_\Omega$-unit vector like $\hat v$, $D_yB^\Omega_{p,\beta}(\hat v)^2\leq 1$, we have proven
\begin{equation}\label{eqn:Jac bound2}
	 |\jac(\tilde \Phi)(y)| \leq \frac{h(F_\Omega)^n}{n^{\frac{n}{2}}} \frac{(\det H)^{\frac{1}{2}}}{|\det K|} \frac{\max_{v\in S_{g_r}(1,y)}F_\Omega(v)^n}{\rho(y,g_r)}.
\end{equation}

If $F_\Omega$ were Riemannian, we could set $g_r = F_\Omega$ and the third term above would be equal to 1. As $F_\Omega$ may not be Riemannian, Boland--Newberger make the following definition:

\begin{definition}[compare with {\cite[p. 3]{boland-newberger}}]\label{defn:distortion}
Let $(Y,F)$ be any Finsler manifold, and let $g_r$ be a Riemannian metric on $Y$. Then let
\[ N(F) = \max_{y\in Y} \max_{v\in S_{g_r}(1,y)} \frac{F(v)^n \vol_{g_r}(B_F(1,y))}{\vol_{g_r}(B_{g_r}(1,y))}. \]
\end{definition}

Note that $\vol_{g_r}(B_{g_r}(1,y))$ is a constant depending only on the dimension of $Y$. It is also easy to check that $N(F)$ is unchanged by scaling the Riemannian metric $g_r$. The following lemma is a straight-forward exercise:

\begin{lemma}\label{lem:distortion 1}
For any Finsler manifold $(Y,F)$ and any $g_r$, $N(F)\geq 1$. Furthermore, $N(F)=1$ if and only if for all $y\in Y$, $F_\Omega$ and the norm induced by $g_r$ are homothetic.
\end{lemma}

Returning to our bounds in equation \eqref{eqn:Jac bound2} and using Definition \ref{defn:distortion}, for all $y \in Y$, 
\begin{equation}\label{eqn:Jac bound1}
	 |\jac(\tilde \Phi)(y)| \leq \frac{h(F_\Omega)^n}{n^{\frac{n}{2}}} \frac{(\det H)^{\frac{1}{2}}}{|\det K|} N(F_\Omega).
\end{equation}

\begin{remark}\label{rmk:gu}
A careful reading of \cite{boland-newberger} shows that, rather than a single Riemannian metric $g_r$, one can run the argument above using a family of Riemannian metrics $\{g_u\}$ parametrized by $F_\Omega$-unit tangent vectors $u$. (The definition of $N(F_\Omega)$ is adjusted accordingly.) We do not exploit this additional flexibility here.

Boland--Newberger use $\{g_u\}$ defined by
\[  g_u(v,w) = \sum_{i,j=1}^n (g_u)_{ij} \ \mbox{ where } \ (g_u)_{ij} = \frac{1}{2} \frac{\partial^2 F^2}{\partial \dot y_i \dot y_j}(y,u).\]
This `direction-dependent' inner product on $T_y\tilde Y$ is a standard tool in Finsler geometry
(see \cite[\S1.2 B]{chern}), but it requires at least $C^2$ regularity of $F^2$, which (as noted in
the Introduction) we do not have unless $(\tilde Y, d_\Omega)$ immediately reduces to the hyperbolic
case 
\cite{benzecri}.
%(\cite{ColboisVerovic}). 
A key step in our argument (see Section \ref{sec:1.1}) is finding a good replacement for $\{g_u\}$.
\end{remark}

The following lemma is where $n\geq 3$ is required for the proof of Theorem \ref{thm:rigidity}. Its proof is an optimization exercise. We remark that volume entropy for the hyperbolic metric, $h(g_0)$, is equal to $n-1$.

\begin{lemma}[see Appendice B, \cite{BCG1}]\label{lem:lin alg lemma}
For a symmetric, $n\times n$ matrix $H$ with trace 1 and $K=I-H$, 
\begin{itemize}
	\item[(i)]  $\frac{(\det H)^{\frac{1}{2}}}{|\det K|} \leq \Big( \frac{\sqrt{n}}{h(g_0)} \Big)^n$ and
	\item[(ii)] if equality holds, $H=\frac{1}{n}I$ and therefore $K=\frac{h(g_0)}{n}I$.
\end{itemize} 
\end{lemma}

Applying Lemma \ref{lem:lin alg lemma} to equation \eqref{eqn:Jac bound1} and then integrating the result over $y \in Y$ gives 
\begin{equation}\label{eqn:result}
	\frac{\vol(\tilde X, g_0)}{\vol(\tilde Y,F_\Omega)} \leq \left( \frac{h(F_\Omega)}{h(g_0)} \right)^n N(F_\Omega)
\end{equation}
which is equivalent to the inequality statement in Theorem \ref{thm:rigidity}.

%
%%
%%%
%%%%%%%%%%%%%%%%%%%%%%%%%%%%%%%%%%%%

\subsection{Rigidity}\label{sec:rigidity}

We now turn to the rigidity part of Theorem \ref{thm:rigidity}. Suppose that equality holds in \eqref{eqn:result}. This forces the equality case of Lemma \ref{lem:lin alg lemma}, i.e. 
\[ K = \frac{h(g_0)}{n} I \ \mbox{ and } \ H = \frac{1}{n} I. \]
Then equation \eqref{eqn:CS} gives
\[  \frac{h(g_0)}{n} | \langle D\tilde \Phi(u),w\rangle | \leq \frac{h(F_\Omega)}{n^{1/2}} \Vert w \Vert_{g_0} \Big( \int_{\beta \in \partial \tilde Y} D_yB_{p,\beta}^\Omega(u)^2 d\mu_y(\beta) \Big)^{\frac{1}{2}}\]
for all $u \in T_y\tilde Y$ and $w\in T_{\tilde \Phi(y)}\tilde X$. Solving for $|\langle D\tilde \Phi(u),w\rangle |$ and taking the supremum over all $w \in T^1_{\tilde \Phi(y)} \tilde X$ gives
\[ \Vert D\tilde\Phi(u) \Vert_{g_0} \leq n^{1/2} \frac{h(F_\Omega)}{h(g_0)} \Big( \int_{\beta \in \partial \tilde Y} D_yB_{p,\beta}^\Omega(u)^2 d\mu_y(\beta) \Big)^{\frac{1}{2}} \]
for all $u \in T\tilde Y$.  

Let $L= (D_y\tilde \Phi)^* \circ (D_y\tilde \Phi)$, where $A^*$ denotes the transpose (with respect to $g_r$ and $g_0$) of a linear map $A:T_y\tilde Y \to T_{\tilde \Phi(y)}\tilde X$. Fix a $g_r$-orthonormal basis $\{u_i\}$. We then calculate:
\begin{align}\label{eqn:trace bound}
\begin{split}
	\tr(L) &= \sum_{i=1}^n g_r(L u_i, u_i) \\
		& = \sum_{i=1}^n \langle D_y\tilde \Phi(u_i),D_y\tilde \Phi(u_i) \rangle \\
		& \leq n \left( \frac{h(F_\Omega)}{h(g_0)}\right)^2 \sum_{i=1}^n \int_{\beta \in \partial \tilde Y} D_yB_{p,\beta}^\Omega(u_i)^2 d\mu_y(\beta)  \\
		& \leq n \left( \frac{h(F_\Omega)}{h(g_0)}\right)^2 \max_{v\in S_{g_r}(1,y)} F_\Omega(v)^2
		\end{split}
\end{align}
where we apply the same reasoning following equation \eqref{eqn:maxFOmega2} to the last line. 
Equality in equation \eqref{eqn:Jac bound2} together with $K = \frac{h(g_0)}{n} I$ and $H = \frac{1}{n} I$ implies
\begin{align}
\begin{split}
	\left( \frac{h(F_\Omega)}{h(g_0)} \right)^{2n} \frac{\max_{v\in S_{g_r}(1,y)}
	F_\Omega(v)^{2n}}{\rho(y,g_r)^2} &= |\jac(\tilde \Phi(y))|^2  \\
		& = \frac{\det L}{\rho(y,g_r)^2} \\
		& \leq \frac{1}{\rho(y,g_r)^2} \left(\frac{\tr(L)}{n}\right)^n \\
		& \leq \left( \frac{h(F_\Omega)}{h(g_0)}\right)^{2n} \frac{\max_{v\in S_{g_r}(1,y)} F_\Omega(v)^{2n}}{\rho(y,g_r)^2} 
		\end{split}
\end{align}
using \eqref{eqn:trace bound}. Equality must hold throughout, in particular when we invoke $(\det L)^{1/n} \leq \tr(L)/n$. Equality implies that 
\[ L = \left( \frac{h(F_\Omega)}{h(g_0)}\right)^2 \max_{v\in S_{g_r}(1,y)} F_\Omega(v)^2 I.\]
Recalling the definition of $L$, this implies that for all $y$, $D_y\tilde \Phi: (T_y\tilde Y,F_\Omega) \to (T_{\tilde \Phi(y)\tilde X}, g_0)$ is an isometry composed with a homothety.

We can now conclude the proof of Theorem \ref{thm:rigidity} with a short argument, using special rigidity properties of Hilbert geometries. That $D_y\tilde\Phi$ is a homothety implies that $S_{F_\Omega}(1,y)$ is an ellipsoid; in particular it is smooth. Using the definition of $F_\Omega$, this implies that $\partial \Omega$ is smooth. Then by 
\cite{benzecri}
%\cite{ColboisVerovic} 
and our choice of normalization for $d_\Omega$, $(\tilde Y,d_\Omega)$ is in fact hyperbolic, so by Mostow's rigidity theorem \cite{mostow} we conclude $(Y, d_\Omega)$ and $(X, g_0)$ are isometric.

%
%%
%%%
%%%%
%%%%%
%%%%%%%%%%%%%%%%%%%%%%%%%%%%%%%%%%%%%%%%

\section{Proofs of Theorems \ref{Main} and \ref{thm:entropytozero}}\label{sec:1.1}

Any properly convex domain $\Omega$ in $\mathbb {RP}^n$ admits a Riemannian metric called the Blaschke metric (cf. \cite[definition 2.1]{benoisthulin13}). The Blaschke metric is projectively invariant and agrees with the Hilbert metric if $\Omega$ is an ellipsoid, which is the case when $(\Omega,F_\Omega)$ is isometric to hyperbolic $n$-space \cite[Proposition 1.6]{tholozan}. Let $F^H_\Omega$ and $F^B_\Omega$ denote the Hilbert and Blaschke norms on $\Omega$, respectively.
%, associated to the metrics denoted $d^H_\Omega$ and $d^B_\Omega$. We write $\vol_{d_\Omega^\ast}$ for the volumes on $\Omega$ associated to the two metrics, and $\vol_{F_\Omega^\ast}$ for the volumes on the tangent space at a point. 

\begin{theorem}[{\cite[Proposition 3.4]{benoisthulin13}}]\label{thm:benoisthulin}
Given any properly convex domain $\Omega$ in $\mathbb {RP}^n$, there exists a constant $K_n\geq 1$
depending only on $n$ such that for all $v\in T\Omega$, 
\[ \frac1{K_n}F_\Omega^H(v)\leq F_\Omega^B(v)\leq K_n F_\Omega^H(v). \]
\end{theorem}

The eccentricity factor of Definition \ref{defn:distortion} for the Hilbert metric with respect to the Blaschke metric is
\[  N(F_\Omega)= \max_{y\in Y}\max_{v\in S_{F_\Omega^B}(1,y)} \frac{F_\Omega^H(v)^n \vol_{F^B_\Omega}(B_{F_\Omega^H}(1,y))}{\vol_{F^B_\Omega}(B_{F^B_\Omega}(1,y))}. \]

It follows that 
\begin{align}
	N(F_\Omega) \leq K_n^{2n} \label{eqn:benoisthulin}
\end{align}
since, by Theorem \ref{thm:benoisthulin}, $B_{F_\Omega^H}(1,y)\subset B_{F_\Omega^B}(K_n,y)$ and
$S_{F_\Omega^B}(1,y)\subset B_{F_\Omega^H}(K_n,y)$ for all $y\in \Omega$.

\begin{proof}[Proof of Theorem \ref{Main}]
By Theorem \ref{thm:rigidity},
\[ 
	N(F_\Omega^H) h(F_\Omega^H)^n\vol(Y,F_\Omega^H) \geq h(g_0)^n \vol(X,g_0) 
\]
with equality only when $(Y,F_\Omega^H)$ and $(X,g_0)$ are isometric. Moreover, since $(X,g_0)$ has constant curvature $-1$, $h(F_\Omega^H)\leq n-1= h(g_0)$ \cite{crampon09}.  Thus,
\[
	\vol(Y,F_\Omega^H) \geq \frac1{N(F_\Omega^H)}\left(\frac{h(g_0)}{h(F_\Omega^H)}\right)^n \vol(X,g_0) \geq \frac{1}{K_n^{2n}} \vol(X,g_0). 
\]
Theorem \ref{Main} follows.

% {\color{red}since $K_n$ is necessarily greater than or equal to 1 ($K_n=1$ can only hold if the geometry is rigid in the space of convex projective structures).  Is this portion necessary? -- DC}

\end{proof}

Theorem \ref{thm:entropytozero} immediately follows from equation \eqref{eqn:benoisthulin} and Theorem \ref{thm:rigidity} for dimensions $n\geq 3$ since $h(g_0)\vol(X,g_0)$ is constant. We treat the $n=2$ case separately in the next section. 

%
%%
%%%
%%%%
%%%%%
%%%%%%%%%%%%%%%%%%%%%%%%%%%%%%%%

\subsection{Entropy and volume in dimension 2}\label{sec:dim2}

\begin{theorem} \label{thm:volblowsup}
Let $Y_t=\Omega_t/\Gamma_t$ be a family of convex projective manifolds homeomorphic to a closed surface $\Sigma$ of negative Euler characteristic. Then
\[ h(F_{\Omega_t}^H) \to 0 \ \Rightarrow \ \vol(Y_t, F_{\Omega_t}^H) \to \infty.\]
\end{theorem}

The approach is comparison with the Riemannian Blascke metric $d_B$ on $\Omega$ via Theorem \ref{thm:benoisthulin}, and an application of Katok's entropy-rigidity theorem \cite[Theorem B]{katok_entropy}.

In this section, we write $\vol_{d_\Omega^\ast}$ for the volumes \emph{on $\Omega$} induced by the two metrics $F_\Omega^\ast$, and we write $\vol_{F_\Omega^\ast}$ for the volumes these metrics induce \emph{on the tangent space at a point}.

\begin{lemma}  \label{lem:comparisons}
There is a constant $V_n$ depending only on dimension such that for any measurable set $U\subset \Omega$, 
\begin{equation}
	\frac1{V_n}\vol_{d_\Omega^B}(U)\leq \vol_{d_\Omega^H}(U)\leq V_n\vol_{d_\Omega^B}(U). \label{eqn:volumes}
\end{equation}
Moreover, we have 
\begin{equation}
	h(F_\Omega^B) \leq K_n h(F_\Omega^H) \label{eqn:entropies}
\end{equation}
where $K_n$ is as in Theorem \ref{thm:benoisthulin}. 
\end{lemma}

\begin{proof}
Since the Blaschke metric is Riemannian we may use it for the definition of the Finsler volume. Then 
\[ \vol_{d_\Omega^H}(U)=\int_U \frac{\vol_{F_\Omega^B}(B_{F_\Omega^B}(1,x))}{\vol_{F_\Omega^B}(B_{F_\Omega^H}(1,x))} d\vol_{F_\Omega^B}(x). \]
By Theorem \ref{thm:benoisthulin} and basic properties of any volume form on $T_x\Omega$, 
\begin{align*}
	\frac1{K^n_n} \vol_{F_\Omega^B}(B_{F_\Omega^B}(1,x)) =   \vol_{F_\Omega^B}(B_{F_\Omega^B}(\frac{1}{K_n},x)) \leq \vol_{F_\Omega^B}(B_{F_\Omega^H}(1,x))   \\ 
	\leq \vol_{F_\Omega^B}(B_{F_\Omega^B}(K_n,x))=K^n_n\vol_{F_\Omega^B}(B_{F_\Omega^B}(1,x)).
\end{align*}
Equation \eqref{eqn:volumes} follows with $V_n=K_n^n$.

By Theorem \ref{thm:benoisthulin} and equation \eqref{eqn:volumes},
\begin{align*}
	h(F_\Omega^F) & = \lim_{r\to\infty}\frac{1}{r}\log \vol_{d_\Omega^B}\left(B_{d_\Omega^B}(r,x)\right) \\ 
				& \leq \lim_{r\to\infty}\frac1r\log \vol_{d_\Omega^B}\left(B_{d_\Omega^H}(K_n r,x)\right) \\
    				& \leq K_n \lim_{r\to\infty}\frac{1}{K_n r} \log \vol_{d_\Omega^H}\left(B_{d_\Omega^H}(K_nr,x)\right) = K_n h(F_\Omega^H).
\end{align*}
This proves equation \eqref{eqn:entropies}.
\end{proof}

The last piece needed to prove Theorem \ref{thm:volblowsup} is another result of Benoist and Hulin:

\begin{lemma}[{\cite[Proposition 3.3]{benoisthulin14}}] \label{fct:curvature}
The curvature of the Blaschke metric on a properly convex $\Omega\subset \mathbb R\mathrm P^2$ is bounded between $-1$ and 0. 
\end{lemma}

By \cite{manning}, this implies that topological entropy of the Blaschke geodesic flow of the quotient manifold is
equal to the volume growth entropy of the Blascke metric. 

\begin{proof}[Proof of Theorem \ref{thm:volblowsup}]
Suppose $h(F_{\Omega_t}^H) \to 0$. By equation \eqref{eqn:entropies}, $h(F_{\Omega_t}^B) \to 0$ as well. By Theorem B in \cite{katok_entropy}, 
\[ (h(F_{\Omega_t}^B))^2\vol_{d_\Omega^B}(Y_t) \geq -2\pi E(\Sigma) \]
where $E(\Sigma)$ is the Euler characteristic of $\Sigma$. Therefore, $\vol_{d_\Omega^B}(Y_t) \to \infty$ and $\vol_{d_\Omega^H}(Y_t)\to\infty$ as well by equation \eqref{eqn:volumes}. 
\end{proof}

%\subsection{Standing assumptions and notation} In Sections~\ref{sec:1.9} and \ref{sec:1.1}, $Y=\Omega/\Gamma$ will represent a compact strictly convex projective manifold of dimension greater than or equal to three, as defined in the Introduction. Furthermore, $Y$ is assumed to be homeomorphic to a hyperbolic manifold denoted by $(X,g_0)$.
%
%The set of all vectors in $T_yY$ with norm greater than or equal to 1, with respect to a Finsler norm F or a Riemannian metric $g$, is denoted by $B_F(1,y)$ or $B_g(1,y)$. Similarly, $S_F(1,y)$ or $S_g(1,y)$ will denote the set of vectors in $T_yY$ with norm equal to 1.

\bibliographystyle{alpha}
\bibliography{biblio}

\begin{thebibliography}{CVV04}

\bibitem[AC16]{adeboye_cooper}
Ilesanmi Adeboye and Daryl Cooper.
\newblock The area of convex projective surfaces and {F}ock--{G}oncharov
  coordinates.
\newblock Available at ArXiv:1506.08245, May 2016.

\bibitem[BCG95]{BCG1}
G\'erard Besson, Gilles Courtois, and Sylvestre Gallot.
\newblock Entropies et rigidit\'es des espaces localement sym\'etriques de
  courbure strictement n\'egative.
\newblock {\em Geom. Funct. Anal.}, 5(5):731--799, 1995.

\bibitem[BCG96]{BCG2}
G\'erard Besson, Gilles Courtois, and Sylvestre Gallot.
\newblock Minimal entropy and {M}ostow's rigidity theorems.
\newblock {\em Ergodic Theory and Dynamical Systems}, 16(4):623--649, 1996.

\bibitem[BCS00]{chern}
D.~Bao, S.-S. Chern, and Z.~Shen.
\newblock {\em An Introduction to {R}iemann-{F}insler Geometry}, volume 200 of
  {\em Graduate Texts in Mathematics}.
\newblock Springer, New York, 2000.

\bibitem[Ben60]{benzecri}
Jean~Paul Benzecri.
\newblock Sur les vari\'et\'es localement affines et localement projectives.
\newblock {\em Bulletin de la S. M. F.}, 88:229--332, 1960.

\bibitem[Ben04]{Ben1}
Yves Benoist.
\newblock Convexes divisibles. {I}.
\newblock In {\em Algebraic groups and arithmetic}, pages 339--374. Tata Inst.
  Fund. Res., Mumbai, 2004.

\bibitem[Ben06]{Ben5}
Yves Benoist.
\newblock Convexes hyperboliques et quasiisom\'etries.
\newblock {\em Geom. Dedicata}, 122:109--134, 2006.

\bibitem[BGS85]{Ballman1}
Werner Ballmann, Mikhael Gromov, and Viktor Schroeder.
\newblock {\em Manifolds of nonpositive curvature}, volume~61 of {\em Progress
  in Mathematics}.
\newblock Birkh\"auser Boston Inc., Boston, MA, 1985.

\bibitem[BH13]{benoisthulin13}
Yves Benoist and Dominique Hulin.
\newblock Cubic differentials and finite volume convex projective surfaces.
\newblock {\em Geometry and Topology}, 17(1):595--620, 2013.

\bibitem[BH14]{benoisthulin14}
Yves Benoist and Dominique Hulin.
\newblock Cubic differentials and hyperbolic convex sets.
\newblock {\em Journal of Differential Geometry}, 98(1):1--19, 2014.

\bibitem[BK53]{BK}
Herbert Busemann and Paul~J. Kelly.
\newblock {\em Projective geometry and projective metrics}.
\newblock Academic Press Inc., New York, N. Y., 1953.

\bibitem[BN01]{boland-newberger}
Jeff Boland and Florence Newberger.
\newblock Minimal entropy rigidity for {F}insler manifolds of negative flag
  curvature.
\newblock {\em Ergodic Theory and Dynamical Systems}, 21(1):13--23, 2001.

\bibitem[CF03a]{CFdegree}
Christopher Connell and Benson Farb.
\newblock The degree theorem in higher rank.
\newblock {\em Journal of Differential Geometry}, 65(1):19--59, 2003.

\bibitem[CF03b]{CFproducts}
Christopher Connell and Benson Farb.
\newblock Minimal entropy rigidity for lattices in products of rank one
  symmetric spaces.
\newblock {\em Comm. Anal. Geom.}, 11(5):1001--1026, 2003.

\bibitem[CF03c]{CFrecent}
Christopher Connell and Benson Farb.
\newblock Some recent applications of the barycenter method in geometry.
\newblock In {\em Topology and Geometry of Manifolds}, volume~71 of {\em Proc.
  Sympos. Pure Math.}, pages 19--50. American Math Society, 2003.

\bibitem[CLT07]{CooperLongThistlethwaite2008}
D.~Cooper, D.~D. Long, and M.~B. Thistlethwaite.
\newblock Flexing closed hyperbolic manifolds.
\newblock {\em Geometry \& Topology}, 11:2413--2440, 2007.

\bibitem[CLT15]{CLTi}
D.~Cooper, D.~D. Long, and S.~Tillmann.
\newblock On convex projective manifolds and cusps.
\newblock {\em Adv. Math.}, 277:181--251, 2015.

\bibitem[Cra09]{crampon09}
Micka{\"e}l Crampon.
\newblock Entropies of strictly convex projective manifolds.
\newblock {\em J. Mod. Dyn.}, 3(4):511--547, 2009.

\bibitem[Cra11]{cramponthese}
Micka{\"e}l Crampon.
\newblock {\em Dynamics and entropies of {H}ilbert metrics}.
\newblock Institut de Recherche Math\'ematique Avanc\'ee, Universit\'e de
  Strasbourg, Strasbourg, 2011.
\newblock Th{\`e}se, Universit{\'e} de Strasbourg, Strasbourg, 2011.

\bibitem[Cra13]{cramponnotes}
Micka\"el Crampon.
\newblock The boundary of a divisible convex set.
\newblock {\em Publ. Mat. Urug.}, 14:105--119, 2013.

\bibitem[CVV04]{CVV}
B.~Colbois, C.~Vernicos, and P.~Verovic.
\newblock L'aire des triangles id\'eaux en g\'eom\'etrie de {H}ilbert.
\newblock {\em Enseign. Math.}, 50(3-4):203--237, 2004.

\bibitem[dlH93]{H}
Pierre de~la Harpe.
\newblock On {H}ilbert's metric for simplices.
\newblock In {\em Geometric group theory, {V}ol.\ 1 ({S}ussex, 1991)}, volume
  181 of {\em London Math. Soc. Lecture Note Ser.}, pages 97--119. Cambridge
  Univ. Press, Cambridge, 1993.

\bibitem[Fer96]{Feres}
Renato Feres.
\newblock The minimal entropy theorem and {M}ostow rigidity: after {G}.
  {B}esson, {G}. {C}ourtois amd {S}. {G}allot.
\newblock (Available at http://www.math.wustl.edu/{$\sim$}feres/mostow.pdf),
  1996.

\bibitem[Gol90]{goldman90}
William~M. Goldman.
\newblock Convex real projective structures on compact surfaces.
\newblock {\em J. Differential Geom.}, 31(3):791--845, 1990.

\bibitem[JM87]{johnsonmillson}
Dennis Johnson and John~J. Millson.
\newblock Deformation spaces associated to compact hyperbolic manifolds.
\newblock In {\em Discrete groups in geometry and analysis ({N}ew {H}aven,
  {C}onn., 1984)}, volume~67 of {\em Progr. Math.}, pages 48--106. Birkh\"auser
  Boston, Boston, MA, 1987.

\bibitem[Kai90]{Kaimanovich1990}
Vadim~A. Kaimanovich.
\newblock Invariant measures of the geodesic flow and measures at infinity on
  negatively curved manifolds.
\newblock 53(4):361--393, 1990.

\bibitem[Kap07]{Kapovich2007}
Michael Kapovich.
\newblock {\em Geometry \& Topology}, 11(3), 2007.

\bibitem[Kat82]{katok_entropy}
A.~Katok.
\newblock Entropy and closed geodesics.
\newblock {\em Ergodic Theory and Dynamical Systems}, 2:339--367, 1982.

\bibitem[Leu06]{leuzinger}
Enrico Leuzinger.
\newblock Entropy of the geodesic flow for metic spaces and {B}ruhat-{T}its
  buildings.
\newblock {\em Adv. Geom.}, 6:475--491, 2006.

\bibitem[Man79]{manning}
Anthony Manning.
\newblock Topological entropy for geodesic flows.
\newblock {\em Annals of Mathematics}, 110(3):567--573, 1979.

\bibitem[Mos68]{mostow}
G.D. Mostow.
\newblock Quasi-conformal mappings in $n$-space and the rigidity of hyperbolic
  space forms.
\newblock {\em Publications math\'ematiques de l'I.H.\'E.S.}, 34:53--104, 1968.

\bibitem[Nie15]{nie}
Xin Nie.
\newblock On the {H}ilbert geometry of simplicial {T}its sets.
\newblock {\em Ann. Inst. Fourier (Grenoble)}, 65(3):1005--1030, 2015.

\bibitem[Pat76]{patterson}
S.~J. Patterson.
\newblock The limit set of a {F}uchsian group.
\newblock {\em Acta Math.}, 136(3-4):241--273, 1976.

\bibitem[Rob03]{roblin}
Thomas Roblin.
\newblock Ergodicit\'e et \'equidistribution en courbure n\'egative.
\newblock {\em M\'em. Soc. Math. Fr. (N.S.)}, (95):vi+96, 2003.

\bibitem[Sul79]{sullivan79}
Dennis Sullivan.
\newblock The density at infinity of a discrete group of hyperbolic motions.
\newblock {\em Inst. Hautes \'Etudes Sci. Publ. Math.}, (50):171--202, 1979.

\bibitem[Tho17]{tholozan}
Nicholas Tholozan.
\newblock Entropy of {H}ilbert metrics and length spectrum of {H}itchin
  representations in $\text{PSL}(3,{R})$.
\newblock (unpublished)
  \href{https://arxiv.org/abs/1506.04640}{https://arxiv.org/abs/1506.04640},
  2017.

\bibitem[Zha15]{zhang15}
Tengren Zhang.
\newblock Degeneration of {H}itchin representations along internal sequences.
\newblock {\em Geom. Funct. Anal.}, 25(5):1588--1645, 2015.

\end{thebibliography}

\end{document}